\documentclass[reqno,12pt,letterpaper]{amsart}
\usepackage[proof]{sdmacros}
\title{Dynamical zeta functions for Axiom A flows}
\author{Semyon Dyatlov}
\email{dyatlov@math.berkeley.edu}
\address{Department of Mathematics, University of California, Berkeley, CA 94720}
\author{Colin Guillarmou}
\email{cguillar@math.cnrs.fr}
\address{Laboratoire de Math\'ematiques d'Orsay,
Facult\'e des Sciences d'Orsay Universit\'e Paris-Sud,
F-91405 Orsay Cedex}
\newtheorem*{theo*}{Theorem}
\newtheorem{exam}[prop]{Example}

\begin{document}

\begin{abstract}
We show that the Ruelle zeta function of any smooth Axiom A flow with orientable
stable/unstable spaces has a meromorphic
continuation to the entire complex plane. The proof uses the meromorphic continuation
result of~\cite{rnc} together with~\cite{ConleyEaston,GMT} which imply that
every basic hyperbolic set can be put into the framework of~\cite{rnc}.
\end{abstract}

\maketitle

%%%%%%%%%%%%%%%%%%%%%%%%%%%%%%%%%%%%%%%%%%%%%%%%%%%%%%%%%%%%%%%%%%%%%%%%%%%%%%%%
%                                 INTRODUCTION                                 %
%%%%%%%%%%%%%%%%%%%%%%%%%%%%%%%%%%%%%%%%%%%%%%%%%%%%%%%%%%%%%%%%%%%%%%%%%%%%%%%%
\addtocounter{section}{1}
\addcontentsline{toc}{section}{1. Introduction}

Let $\mathcal M$ be a compact manifold and $X$ be a $C^\infty$ vector field on $\mathcal M$.
Denote by $\varphi^t=\exp(tX)$ the corresponding flow.
In this note we prove
%%%%%%%%%%%%%%%%%%%%%%%%%%%%%%%%%%%%%%%%%%%%%%%%%%%%%%%%%%%%%%%%%%%%%%%%%%%%%%%%
\begin{theo*}
Assume that $\varphi^t$ is an Axiom A flow
(see Definition~\ref{d:axiom-a}) with orientable
stable/unstable spaces.
Define the \textbf{Ruelle zeta function}
\begin{equation}
  \label{e:zeta-r}
\zeta(\lambda)=\prod_{\gamma^\sharp}\big(1-e^{-\lambda T_{\gamma^\sharp}}\big),\quad
\Re\lambda\gg 1,
\end{equation}
where the product is taken over the primitive closed orbits $\gamma^\sharp$ of $\varphi^t$,
with the exception of fixed points,
and $T_{\gamma^\sharp}$ are the corresponding periods.
Then $\zeta(\lambda)$ extends meromorphically to $\lambda\in\mathbb C$.
\end{theo*}
%%%%%%%%%%%%%%%%%%%%%%%%%%%%%%%%%%%%%%%%%%%%%%%%%%%%%%%%%%%%%%%%%%%%%%%%%%%%%%%%
We refer the reader to~\S\ref{s:dynam} for examples of Axiom A flows
and for the dynamical systems terminology used here.

\renewcommand*{\thefootnote}{$\dagger$}
The above theorem was first conjectured by Smale~\cite[bottom of page~802]{Smale}.%
\footnote{Strictly speaking, \cite{Smale} considers a subclass of Axiom A flows,
namely perturbations of suspensions of certain Axiom A maps. He says
`I must admit a positive answer [to the question of meromorphy] would be a little shocking!'
}
It was proved in the special case of Anosov flows by Giulietti--Liverani--Pollicott~\cite{GLP}.
Dyatlov--Zworski~\cite{zeta} gave a simple microlocal proof in the Anosov case.
In our previous work~\cite{rnc}, we generalized the microlocal techniques used in~\cite{zeta} to handle the case of \emph{open hyperbolic systems}, defined as locally
maximal hyperbolic sets with a smooth neighborhood on which the flow is strictly convex.
We emphasize that the introduction of microlocal tools and the use of escape functions 
for the study of hyperbolic diffeomorphisms and flows appeared first in the works of Faure-Roy-Sj\"ostrand \cite{FRS} and Faure-Sj\"ostrand \cite{FaSj}; they were inspired by the works of Kitaev \cite{Ki}, Blank-Keller-Liverani \cite{BKL}, Liverani \cite{Li}, Gouezel-Liverani \cite{GoLi}, Baladi-Tsujii \cite{BaTs} that used anisotropic Sobolev/H\"older spaces adapted to the dynamics. 
The meromorphic extension of the zeta function for a particular case of Axiom A flow, namely the Grassmanian extension of a contact Anosov flow, was also proved by Faure-Tsujii \cite{FaTs}.

See the introduction to~\cite{rnc} for an overview of results on dynamical
zeta functions and Pollicott--Ruelle resonances
and the book by Baladi~\cite{BaladiBook} for the related case
of hyperbolic maps.

The proof in this note combines~\cite{rnc} with the results of Conley--Easton~\cite{ConleyEaston} and Guillarmou--Mazzucchelli--Tzou~\cite{GMT} which show that every locally maximal hyperbolic set
has a smooth neighborhood with a strictly convex flow, as well as with Smale's spectral decomposition
for Axiom~A flows.

We note that just like in~\cite{rnc} one can consider a more general Ruelle zeta function
with a potential. In fact, the methods of~\cite{rnc} apply to even more general
dynamical traces and twisted zeta functions associated to the action of $X$ on vector bundles,
see~\cite[\S5.1, Theorem~4]{rnc}. Moreover, the assumption that $X$ is $C^\infty$
can be relaxed to $C^k$ for large $k=k(C_0)$ if we only want to continue the Ruelle zeta function
to a half-plane $\{\Re\lambda\geq -C_0\}$. Finally, the orientability hypothesis holds
in many natural cases (such as geodesic flows on orientable negatively curved manifolds)
and can be removed under certain topological assumptions by using a twisted zeta function,
see~\cite[Appendix~B]{GLP}.

%%%%%%%%%%%%%%%%%%%%%%%%%%%%%%%%%%%%%%%%%%%%%%%%%%%%%%%%%%%%%%%%%%%%%%%%%%%%%%%%
\section{Review of hyperbolic dynamics}
  \label{s:dynam}

In this section we briefly review several standard definitions and facts
from the theory of hyperbolic dynamical systems. We refer the reader
to~\cite{KaHa} for a comprehensive treatment of hyperbolic dynamics.

Assume that $\mathcal M$ is a compact $C^\infty$ manifold without boundary
and $\varphi^t=\exp(tX)$ is a $C^\infty$ flow generated by a vector field~$X$.
%%%%%%%%%%%%%%%%%%%%%%%%%%%%%%%%%%%%%%%%%%%%%%%%%%%%%%%%%%%%%%%%%%%%%%%%%%%%%%%%
\begin{defi}
  \label{d:hyperbolic}
Let $\mathcal K\subset\mathcal M$ be a compact $\varphi^t$-invariant set.
We say that $\mathcal K$ is \textbf{hyperbolic}
for the flow $\varphi^t$, if the generator $X$ of the flow does not vanish on $\mathcal K$
and each tangent space $T_x\mathcal M$, $x\in\mathcal K$, admits a flow/stable/unstable decomposition
$$
T_x\mathcal M=E_0(x)\oplus E_s(x)\oplus E_u(x),\quad
x\in \mathcal K
$$
such that:
\begin{itemize}
\item $E_0(x)=\mathbb R X(x)$;
\item $E_s(x),E_u(x)$ depend continuously on the point $x$;
\item $d\varphi^t(E_s(x))=E_s(\varphi^t(x))$ and $d\varphi^t(E_u(x))=E_u(\varphi^t(x))$
for all $x\in\mathcal K$, $t\in\mathbb R$;
\item for any choice of a continuous norm $|\bullet|$ on the fibers of $T\mathcal M$,
there exist constants $C,\theta>0$ such that for all $x\in\mathcal K$,
\begin{equation}
  \label{e:hyperb}
|d\varphi^t(x)v|\leq Ce^{-\theta|t|}|v|\quad\text{when}\quad
\begin{cases}
t\geq 0,& v\in E_s(x);\\
t\leq 0,& v\in E_u(x).
\end{cases}
\end{equation}
\end{itemize}
We say $\varphi^t$ is an \textbf{Anosov flow} if the whole $\mathcal M$ is hyperbolic.
\end{defi}
%%%%%%%%%%%%%%%%%%%%%%%%%%%%%%%%%%%%%%%%%%%%%%%%%%%%%%%%%%%%%%%%%%%%%%%%%%%%%%%%
A fixed point $x\in\mathcal M$, $X(x)=0$, is called \emph{hyperbolic}
if the differential $\nabla X(x)$ has no eigenvalues on the imaginary axis.
Hyperbolic fixed points are nondegenerate and thus isolated.

We also define the \emph{nonwandering set}:
%%%%%%%%%%%%%%%%%%%%%%%%%%%%%%%%%%%%%%%%%%%%%%%%%%%%%%%%%%%%%%%%%%%%%%%%%%%%%%%%
\begin{defi}\cite[p.~796]{Smale}
We call $x\in\mathcal M$ a \textbf{nonwandering point} if for every neighborhood
$V$ of $x$ and every $T>0$ there exists $t\in\mathbb R$ such that
$|t|\geq T$ and $\varphi^t(V)\cap V\neq\emptyset$.
The set of all nonwandering points is called the \textbf{nonwandering set}.
\end{defi}
%%%%%%%%%%%%%%%%%%%%%%%%%%%%%%%%%%%%%%%%%%%%%%%%%%%%%%%%%%%%%%%%%%%%%%%%%%%%%%%%
The nonwandering set is closed and $\varphi^t$-invariant,
see~\cite[Proposition~3.3.4]{KaHa}. Note that each closed
orbit of $\varphi^t$ lies in the nonwandering set.

We now give the definition of Axiom A flows:
%%%%%%%%%%%%%%%%%%%%%%%%%%%%%%%%%%%%%%%%%%%%%%%%%%%%%%%%%%%%%%%%%%%%%%%%%%%%%%%%
\begin{defi}\cite[\S II.5, (5.1)]{Smale}
  \label{d:axiom-a}
The flow $\varphi^t$ is \textbf{Axiom A} if:
\begin{enumerate}
\item the nonwandering set is the disjoint union of the set $\mathcal F$ of fixed
points and the closure $\mathcal K$ of the union of all closed orbits;
\item all fixed points of $\varphi^t$ are hyperbolic;
\item the set $\mathcal K$ is hyperbolic for the flow $\varphi^t$.
\end{enumerate}
\end{defi}
%%%%%%%%%%%%%%%%%%%%%%%%%%%%%%%%%%%%%%%%%%%%%%%%%%%%%%%%%%%%%%%%%%%%%%%%%%%%%%%%
We also define locally maximal sets and basic hyperbolic sets:
%%%%%%%%%%%%%%%%%%%%%%%%%%%%%%%%%%%%%%%%%%%%%%%%%%%%%%%%%%%%%%%%%%%%%%%%%%%%%%%%
\begin{defi}
We say that a compact $\varphi^t$-invariant set $K\subset \mathcal M$ is \textbf{locally maximal}
for the flow $\varphi^t$, if there exists a neighborhood $V$ of $K$ such that
$$
K=\bigcap_{t\in\mathbb R}\varphi^t(V).
$$
\end{defi}
%%%%%%%%%%%%%%%%%%%%%%%%%%%%%%%%%%%%%%%%%%%%%%%%%%%%%%%%%%%%%%%%%%%%%%%%%%%%%%%%
\begin{defi}\cite[Chapter~9]{ParryPollicott}
A compact $\varphi^t$-invariant set $K\subset\mathcal M$ is called a \textbf{basic hyperbolic set\/}%
\footnote{Strictly speaking,
a single closed orbit is not considered a basic set,
but we ignore this minor detail here.}
for $\varphi^t$ if\begin{enumerate}
\item $K$ is locally maximal for $\varphi^t$;
\item $K$ is hyperbolic for $\varphi^t$;
\item the flow $\varphi^t|_K$ topologically transitive (i.e. contains a dense orbit); and
\item $K$ is the closure of the union of all closed orbits of $\varphi^t|_K$.
\end{enumerate}
\end{defi}
%%%%%%%%%%%%%%%%%%%%%%%%%%%%%%%%%%%%%%%%%%%%%%%%%%%%%%%%%%%%%%%%%%%%%%%%%%%%%%%%
The nonwandering set of an Axiom A flow has the following
\emph{spectral decomposition}:

%%%%%%%%%%%%%%%%%%%%%%%%%%%%%%%%%%%%%%%%%%%%%%%%%%%%%%%%%%%%%%%%%%%%%%%%%%%%%%%%
\begin{prop}\cite[\S II.5, Theorem~5.2]{Smale}
  \label{l:spectral}
Assume that $\varphi^t$ is an Axiom A flow and let $\mathcal K$ be given
by Definition~\ref{d:axiom-a}. Then we can write $\mathcal K$ as a finite disjoint union
$$
\mathcal K=K_1\sqcup\dots\sqcup K_N
$$
where each $K_j$ is a basic hyperbolic set.
\end{prop}
%%%%%%%%%%%%%%%%%%%%%%%%%%%%%%%%%%%%%%%%%%%%%%%%%%%%%%%%%%%%%%%%%%%%%%%%%%%%%%%%
See also~\cite[Exercise~18.3.7]{KaHa} for a proof of the spectral decomposition,
and~\cite[Lemma~3.9]{BowenLec} for the local maximality of the nonwandering set.

We finally give a few examples:
%%%%%%%%%%%%%%%%%%%%%%%%%%%%%%%%%%%%%%%%%%%%%%%%%%%%%%%%%%%%%%%%%%%%%%%%%%%%%%%%
\begin{exam}
  \label{x:simple}
Let $\mathcal M=\mathbb R^2_{x,y}/2\pi \mathbb Z^2$ be the torus and
$\varphi_t=\exp(tX)$ where $X=(\sin x)\partial_x+\partial_y$.
The nonwandering set $\mathcal K$ has two connected components,
each being a single closed orbit:
\begin{equation}
  \label{e:simple-dec}
\mathcal K=\{x=0\}\sqcup\{x=\pi\}.
\end{equation}
The spectral decomposition is given by~\eqref{e:simple-dec}.
Note however that the entire $\mathcal M$ is not hyperbolic for $\varphi^t$.
\end{exam}
%%%%%%%%%%%%%%%%%%%%%%%%%%%%%%%%%%%%%%%%%%%%%%%%%%%%%%%%%%%%%%%%%%%%%%%%%%%%%%%%
\begin{exam}
Let $(M,g)$ be a (possibly noncompact) complete Riemannian manifold
of negative sectional curvature which is asymptotically hyperbolic.
Put $\mathcal M:=SM$, the sphere bundle of $M$,
and let $\varphi^t:\mathcal M\to\mathcal M$ be the geodesic flow.
Let $\mathcal K$ be the union of all geodesics which are trapped,
that is their closures in $SM$ are compact. Then
$\mathcal K$ is a locally maximal hyperbolic set for $\varphi^t$.
(The noncompactness of $\mathcal M$ is not an issue since
$\mathcal K$ is compact and the behavior of the flow
outside of a neighborhood of $\mathcal K$ is irrelevant.)
See~\cite[\S6.3]{rnc} for details and more general examples.
\end{exam}
%%%%%%%%%%%%%%%%%%%%%%%%%%%%%%%%%%%%%%%%%%%%%%%%%%%%%%%%%%%%%%%%%%%%%%%%%%%%%%%%

%%%%%%%%%%%%%%%%%%%%%%%%%%%%%%%%%%%%%%%%%%%%%%%%%%%%%%%%%%%%%%%%%%%%%%%%%%%%%%%%
%%%%%%%%%%%%%%%%%%%%%%%%%%%%%%%%%%%%%%%%%%%%%%%%%%%%%%%%%%%%%%%%%%%%%%%%%%%%%%%%
\section{Proof of the theorem}

We first show meromorphic continuation of the Ruelle zeta function $\zeta_K$
for a locally maximal hyperbolic set:
%%%%%%%%%%%%%%%%%%%%%%%%%%%%%%%%%%%%%%%%%%%%%%%%%%%%%%%%%%%%%%%%%%%%%%%%%%%%%%%%
\begin{prop}
  \label{l:basic-case}
Let $K\subset\mathcal M$ be a locally maximal hyperbolic set for $\varphi^t$.
Define the Ruelle zeta function $\zeta_K$ by~\eqref{e:zeta-r}
where we only take the closed trajectories of $\varphi^t$ which lie in $K$. Then
$\zeta_K$ continues to a meromorphic function on $\mathbb C$.
\end{prop}
%%%%%%%%%%%%%%%%%%%%%%%%%%%%%%%%%%%%%%%%%%%%%%%%%%%%%%%%%%%%%%%%%%%%%%%%%%%%%%%%
The proof of Proposition~\ref{l:basic-case} relies on~\cite{rnc}. To reduce
to the case considered in~\cite{rnc} we use
%%%%%%%%%%%%%%%%%%%%%%%%%%%%%%%%%%%%%%%%%%%%%%%%%%%%%%%%%%%%%%%%%%%%%%%%%%%%%%%%
\begin{lemm}
  \label{l:convex-embed}
Let $K\subset\mathcal M$ be a locally maximal hyperbolic set for $\varphi^t$. Then there exists a neighborhood
$\mathcal U$ of $K$ in $\mathcal M$ with $C^\infty$ boundary and a $C^\infty$ vector field $X_0$ on
$\overline{\mathcal U}$ such that:
\begin{enumerate}
\item the boundary $\partial\mathcal U$ is strictly convex with respect to $X_0$ in the following sense:
\begin{equation}
  \label{e:quadratic-convexity}
\rho(x)=0,\quad
X_0\rho(x)=0\quad\Longrightarrow\quad
X_0^2\rho(x)<0
\end{equation}
where $\rho\geq 0$ is any boundary defining function on $\mathcal U$;
\item $X_0=X$ in a neighborhood of $K$;
\item $K=\bigcap_{t\in\mathbb R}\varphi_0^t(\mathcal U)$ where
$\varphi_0^t:=\exp(tX_0)$ is the flow generated by $X_0$.
\end{enumerate}  
\end{lemm}
%%%%%%%%%%%%%%%%%%%%%%%%%%%%%%%%%%%%%%%%%%%%%%%%%%%%%%%%%%%%%%%%%%%%%%%%%%%%%%%%
\begin{proof}
The proof follows Guillarmou--Mazzucchelli--Tzou~\cite[Lemma~2.3]{GMT}.
We first use a result of Conley--Easton~\cite[Theorem~1.5]{ConleyEaston}.
Since $K$ is locally maximal for $\varphi^t$, there exists an open set $V\subset\mathcal M$
containing $K$ such that $K=\bigcap_{t\in\mathbb R}\varphi^t(\overline V)$.
If $x\in \partial V$, then in particular $x\notin K$, thus there exists
$t\in\mathbb R$ such that $\varphi^t(x)\notin\overline V$.
Therefore $V$ is an isolating neighborhood for $\varphi^t$ in the sense
of~\cite[Definition~1.1]{ConleyEaston}
and $K$ is an isolated invariant set in the sense of~\cite[Definition~1.2]{ConleyEaston}.
Therefore by~\cite[Theorem~1.5]{ConleyEaston} there exists an isolating block,
which is an open subset $\mathcal U\subset \mathcal M$ with $K\subset \mathcal U$ with
the following properties:
\begin{itemize}
\item $\mathcal U$ has compact closure and $C^\infty$ boundary,
that is $\mathcal U=\{x\in\mathcal M\mid \rho(x)>0\}$ for some function
$\rho\in C^\infty(\mathcal M;\mathbb R)$ such that $d\rho\neq 0$ on $\partial \mathcal U=\{x\in\mathcal M\mid \rho(x)=0\}$.
\item the set $\partial_0\mathcal U:=\{x\in\partial\mathcal U\mid X_0\rho(x)=0\}$
is a codimension 1 smooth submanifold of $\partial\mathcal U$;
\item for each $x\in \partial_0\mathcal U$, there exists $\varepsilon>0$
such that
\begin{equation}
  \label{e:topological-convexity}
\varphi^t(x)\notin\overline{\mathcal U}\quad\text{for all}\quad t\in (-\varepsilon,\varepsilon)\setminus \{0\}.
\end{equation}
\end{itemize}
Now the vector field $X_0$ is obtained by modifying $X$ slightly near $\partial_0\mathcal U$
so that we still have $K=\bigcap_{t\in\mathbb R}\varphi_0^t(\overline {\mathcal U})$
and the topological (local) convexity condition~\eqref{e:topological-convexity}
is replaced by the quadratic differential convexity condition~\eqref{e:quadratic-convexity}.
We refer to the proof of~\cite[Lemma~2.3]{GMT} for details.
\end{proof}
%%%%%%%%%%%%%%%%%%%%%%%%%%%%%%%%%%%%%%%%%%%%%%%%%%%%%%%%%%%%%%%%%%%%%%%%%%%%%%%%
We now give
%%%%%%%%%%%%%%%%%%%%%%%%%%%%%%%%%%%%%%%%%%%%%%%%%%%%%%%%%%%%%%%%%%%%%%%%%%%%%%%%
\begin{proof}[Proof of Proposition~\ref{l:basic-case}]
By part~(3) of Lemma~\ref{l:convex-embed}, every closed trajectory of $\varphi^t_0$ in $\overline{\mathcal U}$
lies in $K$ and thus is a closed trajectory of $\varphi^t$. Therefore the Ruelle
zeta function $\zeta_K$ is equal to the Ruelle zeta function of $\varphi^t_0$ on $\overline{\mathcal U}$.
Now $(\mathcal U,\varphi^t_0)$ is an \emph{open hyperbolic system}
in the sense of~\cite[Assumptions~(A1)--(A4)]{rnc}. Therefore~\cite[Theorem~3]{rnc} applies
(with potential set to~0) and gives a meromorphic continuation of $\zeta_K$ to $\mathbb C$.
\end{proof}
%%%%%%%%%%%%%%%%%%%%%%%%%%%%%%%%%%%%%%%%%%%%%%%%%%%%%%%%%%%%%%%%%%%%%%%%%%%%%%%%
The main theorem now follows using the spectral
decomposition theorem, Proposition~\ref{l:spectral}. Indeed,
if $\mathcal K=K_1\sqcup\dots \sqcup K_N$ is the spectral decomposition
of $\varphi^t$, then each closed trajectory of $\varphi^t$ is contained
in one of the sets $K_j$. Therefore the Ruelle zeta function~\eqref{e:zeta-r}
factorizes as
$$
\zeta(\lambda)=\zeta_{K_1}(\lambda)\cdots\zeta_{K_N}(\lambda).
$$
Since each $\zeta_{K_j}$ admits a meromorphic continuation to $\mathbb C$
by Proposition~\ref{l:basic-case}, the function $\zeta$ admits a meromorphic
continuation to $\mathbb C$ as well.

%%%%%%%%%%%%%%%%%%%%%%%%%%%%%%%%%%%%%%%%%%%%%%%%%%%%%%%%%%%%%%%%%%%%%%%%%%%%%%%%
%%%%%%%%%%%%%%%%%%%%%%%%%%%%%%%%%%%%%%%%%%%%%%%%%%%%%%%%%%%%%%%%%%%%%%%%%%%%%%%%
\medskip\noindent\textbf{Acknowledgements.}
This research was conducted during the period SD served as
a Clay Research Fellow.
CG was partially supported by the ANR project ANR-13-JS01-0006
 and by the ERC consolidator grant IPFLOW.

%%%%%%%%%%%%%%%%%%%%%%%%%%%%%%%%%%%%%%%%%%%%%%%%%%%%%%%%%%%%%%%%%%%%%%%%%%%%%%%%

\end{document}